\documentclass[11pt,a4paper]{amsart}
\usepackage[utf8]{inputenc}
\usepackage{amssymb}
\usepackage{amsmath}
\usepackage{amsthm}
\usepackage{xcolor}
\usepackage{hyperref}
\usepackage{tikz-cd}
\usepackage{mathscinet}
\usepackage[top=3 cm, bottom=3 cm, left=3 cm, right=3cm]{geometry}
\usepackage{hyperref}
\usepackage{lipsum}
\usepackage{adjustbox}
    \hypersetup{colorlinks=true,allcolors=red}
    \usepackage[capitalize,nameinlink,noabbrev]{cleveref}

\usepackage[backref=true,giveninits=true,doi=false,url=false,isbn=false,backend=biber]{biblatex}
    \newbibmacro{string+doiurlisbn}[1]{%
      \iffieldundef{doi}{%
        \iffieldundef{url}{%
          \iffieldundef{isbn}{%
            \iffieldundef{issn}{%
              #1%
            }{%
              \href{http://books.google.com/books?vid=ISSN\thefield{issn}}{#1}%
            }%
          }{%
            \href{http://books.google.com/books?vid=ISBN\thefield{isbn}}{#1}%
          }%
        }{%
          \href{\thefield{url}}{#1}%
        }%
      }{%
        \href{http://dx.doi.org/\thefield{doi}}{#1}%
      }%
    }
    \DeclareFieldFormat{title}{\usebibmacro{string+doiurlisbn}{\mkbibemph{#1}}}
    \DeclareFieldFormat[article,incollection,thesis,misc,inproceedings,online,inbook]{title}{\usebibmacro{string+doiurlisbn}{\mkbibquote{#1}}}
    \DeclareFieldFormat{year}{\usebibmacro{year-parenthesis}{\mkbibemph{#1}}}
\renewbibmacro*{date}{%
\iffieldundef{year}
{}
{\ifentrytype{misc}{\printtext[parens]{\printdate}}{\printdate}}}%
\hypersetup{
           breaklinks=true,   
           colorlinks=true,   
           pdfusetitle=true,  
        }

\newcounter{dummy}
\numberwithin{dummy}{section}
\newtheorem{thm}[dummy]{Theorem}

\newtheorem{lem}[dummy]{Lemma}
\newtheorem{prop}[dummy]{Proposition}
\newtheorem{cor}[dummy]{Corollary}
\theoremstyle{definition}
\newtheorem{rmk}[dummy]{Remark}

\numberwithin{equation}{section}

\newcommand{\noop}[1]{}

\DeclareMathOperator{\End}{End}

\title[Algebraic relations that preserve endomorphism rings of CM $j$-invariants]{Algebraic relations over finite fields that preserve the endomorphism rings of CM $j$-invariants}

\author[F. Campagna]{Francesco Campagna}
\address[F. Campagna]{Université Clermont Auvergne - LMBP UMR 6620 - CNRS, Campus des Cézeaux 3, place Vasarely 63178 Aubière cedex, France}
\email{francesco.campagna@uca.fr}

\author[G. A. Dill]{Gabriel A. Dill}
\address[G. A. Dill]{Rheinische Friedrich-Wilhelms-Universit\"at Bonn, Mathematisches Institut, Endenicher Allee 60, 53115 Bonn, Germany}
\email{dill@math.uni-bonn.de}

\date{\today}

\begin{document}

\renewcommand{\contentsname}{Capitvla}
\renewcommand{\refname}{\textsc{Bibliographia}}
\renewcommand{\abstractname}{Epitome}


\maketitle
\begin{abstract}
    We characterise the integral affine plane curves over a finite field $k$ with the property that all but finitely many of their $\overline{k}$-points have coordinates that are $j$-invariants of elliptic curves with isomorphic endomorphism rings. This settles a finite field variant of the André-Oort conjecture for $Y(1)^2_\mathbb{C}$, which is a theorem of André. We use our result to solve the modular support problem for function fields of positive characteristic.
\end{abstract}

\tableofcontents

\section{\textsc{Introdvctio}}

If an integral affine plane curve $\mathcal{C}$ defined over a number field $K$ contains infinitely many points whose coordinates are roots of unity, what can one say about $\mathcal{C}$? In particular, is $\mathcal{C}$ necessarily the Zariski closure in $\mathbb{A}_K^2$ of an irreducible component of an algebraic subgroup of $\mathbb{G}^2_{m,K}$, the square of the multiplicative group? This question of Lang was answered in the affirmative by Ihara, Serre, and Tate \cite{Lang_1965}. Similarly, as a first step towards the André-Oort conjecture, of which a proof has in the meantime been announced \cite{Pila_Tsimerman_2014,Daw_Orr_2016,Klingler_Ullmo_Yafaev_2016,Tsimerman_2018,Binyamini_Schmidt_Yafaev_2023,Pila_Shankar_Tsimerman_2022}, André \cite{Andre_1998} settled a modular analogue of this question where now roots of unity are replaced by $j$-invariants of elliptic curves with complex multiplication (CM). In this case, the curve $\mathcal{C}$ must either be defined by a modular polynomial parametrising pairs of elliptic curves related by a cyclic isogeny of some fixed degree or be a vertical or horizontal line.

When one replaces the number field $K$ with a finite field $\mathbb{F}_q$ of characteristic $p \in \mathbb{N}= \{1,2,\hdots\}$, the questions above become trivial. Indeed, with finitely many exceptions, all elements of a fixed algebraic closure $F$ of $\mathbb{F}_q$ are of finite multiplicative order as well as $j$-invariants of CM elliptic curves. Thus, all but finitely many $F$-points of any integral affine plane curve $\mathcal{C}_{/\mathbb{F}_q}$ that is not a vertical or horizontal line will have coordinates that are both roots of unity and CM $j$-invariants. But what happens if all but finitely many points in $\mathcal{C}(F)$ have coordinates that are both roots of unity \textit{of the same order}? Or, in the modular setting, what if all but finitely many points in $\mathcal{C}(F)$ have coordinates that are $j$-invariants of elliptic curves with complex multiplication \textit{by the same imaginary quadratic order}? While an answer to the first question can be deduced from a result of Scanlon and Voloch \cite{Scanlon_Voloch} (see Section \ref{sec:multiplicative case}), an answer to its modular counterpart does not seem to follow directly from the literature. Such an answer constitutes the main result of this paper.

\begin{thm} \label{thm:main_thm}
    Let $p \in \mathbb{N}$ be a prime and let $F$ denote an algebraic closure of the finite field $\mathbb{F}_p$. Let $\mathcal{C} \subseteq \mathbb{A}^2_F$ be an integral closed curve satisfying the following property: for all but finitely many points $(x_1,x_2) \in \mathcal{C}(F)$ we have that $x_1$ and $x_2$ are $j$-invariants of elliptic curves with isomorphic endomorphism rings. Then the coordinate functions $X,Y$ on $\mathcal{C}$ satisfy either $X=Y^{p^n}$ or $Y=X^{p^n}$ for some $n \in \mathbb{Z}_{\geq 0}$.
\end{thm}

The attentive reader has certainly noticed that the above theorem is not completely analogous to the theorem of André in characteristic 0. To obtain a better parallel, one would have to replace the expression ``for all but finitely many" in the statement with ``for infinitely many". However, using the $ABC$ theorem for function fields, one can show that the conclusion of Theorem \ref{thm:main_thm} does not hold anymore under this weaker hypothesis. In fact, any integral affine plane curve over $F$ that is neither a vertical nor a horizontal line contains infinitely many points whose coordinates are $j$-invariants of elliptic curves with isomorphic endomorphism rings. The same is true for points whose coordinates both have the same multiplicative order. This is discussed in more detail in Section \ref{sec:final section}. 

In order to prove Theorem \ref{thm:main_thm}, we first show that any two elliptic curves $E_1$ and $E_2$ over the function field of $\mathcal{C}$ with $j$-invariants $X$ and $Y$ respectively are geometrically isogenous to each other. We then prove that this isogeny can be chosen to be a power of Frobenius.

In the beginning, our investigation of these matters was motivated by the pursuit of a solution to the so-called \textit{modular support problem} for function fields of positive characteristic. Let us explain.

In 1988, while attending a conference in Banff, Erd\H{o}s asked whether it is true that, given two positive integers $a,b$ with the property that for every $n\in \mathbb{N}$ the set of primes dividing $a^n-1$ is equal to the set of primes dividing $b^n-1$, it follows that $a=b$. The answer is yes and this follows from the work of Schinzel \cite{Schinzel_1960}. A complete answer to an analogous question over arbitrary number fields was first given by Corrales and Schoof \cite{Corrales-Rodriganez_Schoof_1997}, who work under the less stringent assumption that the support of $a^n-1$, \emph{i.e.} the set of prime ideals dividing $a^n-1$, is only contained in the one of $b^n-1$ (rather than equal to it) and deduce that $b = a^m$ for some $m \in \mathbb{Z}$. Subsequently, several authors generalised and solved this problem in the broader context of abelian and split semi-abelian varieties, see  \cite{BGK_2003, Khare_Prasad_2004, Larsen_2003, Perucca_2009}.

Following \cite[Section 4]{Campagna_Dill}, one can formulate more general support problems over arbitrary Dedekind domains as follows: let $R$ be a Dedekind domain which is not a field and let $\mathcal{N}$ be an arbitrary countably infinite set. We are given a polynomial $f_n(T) \in R[T]$ for each $n \in \mathcal{N}$ and two elements $a,b \in R$. In this setting, the support problem asks us to understand how $a$ and $b$ are related if we know that for all but finitely many $n \in \mathcal{N}$, every prime ideal factor of $f_n(a)$ also divides $f_n(b)$. The most natural Dedekind domains to consider for concrete instances of the problem are certainly rings of $S$-integers in number fields and coordinate rings of smooth affine irreducible curves defined over fields of any characteristic. If $f_n(T)=T^n-1$ for all $n \in \mathcal{N} := \mathbb{N}$ as in the original question of Erd\H{o}s, we talk of the \textit{multiplicative support problem}. If $f_n(T)=\Psi_n(T)$ is the family of cyclotomic polynomials ($n \in \mathbb{N}$), we talk of the \textit{cyclotomic support problem}.

In \cite{Campagna_Dill}, motivated by questions arising in the context of arithmetic unlikely intersections, we were led to examine a modular variant of the support problem where the family of polynomials under consideration is the family of Hilbert class polynomials $H_D(T)$ for $D$ varying in the set $\mathbb{D}$ of negative integers congruent to $0$ or $1$ modulo $4$. By definition, $H_D(T) \in \mathbb{Z}[T]$ is the minimal polynomial of any $j$-invariant of a complex elliptic curve with complex multiplication by the unique quadratic order of discriminant $D$. The polynomial $H_D(T)$ can be viewed as a polynomial with coefficients in an arbitrary ring $R$ by means of the unique ring homomorphism $\mathbb{Z} \to R$. The \textit{modular support problem} has been solved almost completely in \cite{Campagna_Dill} in the case where $R$ is either a ring of $S$-integers in a number field or the coordinate ring of a smooth affine irreducible curve defined over an algebraically closed field of characteristic 0. Using Theorem \ref{thm:main_thm}, we can now settle the case where $R$ is the coordinate ring of a smooth affine irreducible curve defined over $F$, an algebraic closure of $\mathbb{F}_p$ as above.

\begin{thm} \label{thm:main_theorem}
    Let $R$ be the coordinate ring of a smooth affine irreducible curve $\mathcal{C}_{/F}$ and let $A, B \in R\setminus F$. Suppose that for all prime ideals $\mathfrak{p}$ of $R$ and for all but finitely many $D \in \mathbb{D}$ the implication
    \[
        \mathfrak{p} \mid H_D(A) \Rightarrow \mathfrak{p} \mid H_D(B)
    \]
    holds. Then there exists a non-negative integer $n$ such that $A=B^{p^n}$ or $B=A^{p^n}$.
\end{thm}

Theorem \ref{thm:main_theorem} follows at once from applying Theorem \ref{thm:main_thm} to the Zariski closure of the image of the morphism $\mathcal{C} \to \mathbb{A}_F^2$ defined by the pair $(A,B)$. We provide a complete proof in Section \ref{sec:final section}.

This article is organised as follows: in Section \ref{sec:multiplicative case}, we answer our finite field variant of Lang's question. This also yields a solution to the multiplicative and cyclotomic support problems for the coordinate ring of a smooth affine irreducible curve defined over $F$, see Theorem \ref{thm:multiplicative}. In Section \ref{sec:recordantiae}, we review some relevant background on elliptic curves, which we then use in Section \ref{sec:apologeticum}, where we prove Theorem \ref{thm:main_thm}. Finally, in Section \ref{sec:final section}, we prove Theorem \ref{thm:main_theorem} and we show that every integral affine plane curve over $F$ that is neither a vertical nor a horizontal line contains infinitely many $F$-points whose coordinates are both roots of unity of the same order as well as $j$-invariants of elliptic curves with isomorphic endomorphism rings.

\section{\textsc{Variatio} -- a finite field variant of Lang's question} \label{sec:multiplicative case}

The main result of this section, Theorem \ref{thm:multiplicative}, solves the finite field variant of the original question of Lang about roots of unity on plane curves. Theorem \ref{thm:multiplicative} is obtained as a consequence of the work of Scanlon and Voloch \cite{Scanlon_Voloch}. Let us recall the main points of this work that we will use.

Let $p \in \mathbb{Z}$ be a prime and fix an algebraic closure $F$ of the finite field $\mathbb{F}_p$. Given a prime $\ell \neq p$ and $n \in \mathbb{Z}_{\geq 0}$, we denote by $\mu_{\ell^n} \subseteq F$ the group of roots of unity of order dividing $\ell^n$ and we set $\mu_{\ell^\infty} = \bigcup_{m \in \mathbb{N}}{\mu_{\ell^m}} \subseteq F$. It is not difficult to prove that there exist infinitely many $\sigma \in \mathrm{Gal}(\mathbb{F}_p(\mu_{\ell^\infty})/\mathbb{F}_p)$ such that $\sigma(\zeta)=\zeta^a$ for all $\zeta \in \mu_{\ell^{\infty}}$ where $a=a(\sigma) \in \mathbb{Z}$ is coprime to $p$. One can see this for instance as follows.

Choose an integer $n \in \mathbb{Z}_{>0}$ such that $\mathbb{F}_p(\mu_{\ell^m}) \neq \mathbb{F}_p(\mu_{\ell^{m+1}})$ for all integers $m \geq n$. Then for every $m\geq n$ we have 
\[
\mathrm{Gal}\left(\mathbb{F}_p(\mu_{\ell^m})/\mathbb{F}_p(\mu_{\ell^n}) \right) \simeq 1+\ell^n \cdot \left(\mathbb{Z}/\ell^m \mathbb{Z}\right)
\]
where an element $y \in 1+\ell^n \cdot \left(\mathbb{Z}/\ell^m \mathbb{Z}\right)$ corresponds to the automorphism acting on $\mu_{\ell^m}$ as $\zeta \mapsto \zeta^y$. Taking inverse limits, one gets the isomorphism
\[
\mathrm{Gal}\left(\mathbb{F}_p(\mu_{\ell^\infty})/\mathbb{F}_p(\mu_{\ell^n}) \right) \simeq 1 + \ell^n \mathbb{Z}_\ell
\]
where an element $y \in 1 + \ell^n \mathbb{Z}_\ell$ corresponds to the automorphism acting on $\mu_{\ell^m}$ as $\zeta \mapsto \zeta^{y \text{ mod } \ell^m}$ for all $m \geq n$. Take now any $a \in \mathbb{Z}$ which is not divisible by $p$ and such that $a \equiv 1 \text{ mod } \ell^n$. Set $x = (a-1)/\ell^n \in \mathbb{Z} \subseteq \mathbb{Z}_\ell$. Then the automorphism $\sigma \in \mathrm{Gal}\left(\mathbb{F}_p(\mu_{\ell^\infty})/\mathbb{F}_p(\mu_{\ell^n}) \right)$ corresponding to $1+\ell^n x$ acts on $\mu_{\ell^\infty}$ as $\zeta \mapsto \zeta^a$ as we wanted to show.

Let now $X \subseteq \mathbb{G}_{m,F}^2$ be a closed integral curve. Suppose that there exist infinitely many points $(\zeta,\zeta') \in X(F)$ with $\zeta, \zeta' \in \mu_{\ell^\infty}$ for some fixed prime $\ell \neq p$. By the discussion above, we can find $\tau \in \mathrm{Gal}(F/\mathbb{F}_p)$ such that $\tau(\eta)=\eta^a$ for all $\eta \in \mu_{\ell^\infty}$ where $a\in \mathbb{N}$ is coprime to $p$. We can then apply \cite[Proposition 3]{Scanlon_Voloch} with, using the notation of \cite{Scanlon_Voloch}, $g = 2$ and $(L,\sigma)$ a difference closed field such that $F \subseteq L$ and $\sigma\lvert_F = \tau$ (such a difference closed field exists by \cite[Theorem on p.~3007]{Chatzidakis_Hrushovski}) to deduce the following result.

\begin{thm}[Scanlon-Voloch] \label{thm:scanlonvoloch}
    Let $X \subseteq \mathbb{G}_{m,F}^2$ be a closed integral curve and suppose that there exists a prime $\ell \neq p$ for which there exist infinitely many points $(\zeta,\zeta') \in X(F)$ with $\zeta, \zeta' \in \mu_{\ell^\infty}$. Then $X$ is an irreducible component of an algebraic subgroup of $\mathbb{G}_{m,F}^2$.
\end{thm}

The above theorem now implies the following corollary.

\begin{cor} \label{cor:multiplicative}
Let $\mathcal{C} \subseteq \mathbb{G}_{m,F}^2$ be a closed integral curve. Suppose that for all but finitely many points $(x_1,x_2) \in \mathcal{C}(F)$ the multiplicative order of $x_2$ divides the multiplicative order of $x_1$. Then $\mathcal{C}$ is an irreducible component of an algebraic subgroup of $\mathbb{G}_{m,F}^2$.
\end{cor}

\begin{proof}
If the first coordinate projection $\mathcal{C} \to \mathbb{G}_{m,F}$ is constant on $\mathcal{C}$, then we have that $\mathcal{C} = \{Q\} \times \mathbb{G}_{m,F}$ for some $Q \in \mathbb{G}_{m,F}(F)$. Since every point in $\mathbb{G}_{m,F}(F)$ is torsion, the curve $\mathcal{C}$ is an irreducible component of an algebraic subgroup of $\mathbb{G}_{m,F}^2$.

Hence, we can assume that the first coordinate projection is not constant on $\mathcal{C}$. Thanks to \cite[Proposition 15.4~(1)]{GoertzWedhorn}, applied to the first coordinate projection, there exists $N_0 \in \mathbb{N}$ such that for every root of unity $\zeta$ of order $n > N_0$ there exists some $\zeta' \in F^{\ast}$ such that $(\zeta,\zeta') \in \mathcal{C}(F) \subseteq \mathbb{G}^2_{m,F}(F) \simeq (F^{\ast})^2$.

In particular, this holds if $n$ is any positive power of a sufficiently large prime $\ell \neq p$. After choosing a larger $\ell$ if necessary, we can deduce from our hypothesis that the order of $\zeta'$ is a non-negative power of $\ell$. Varying the power of $\ell$ shows that $\mathcal{C}$ contains infinitely many points with coordinates in $\mu_{\ell^\infty}$. By Theorem \ref{thm:scanlonvoloch}, the curve $\mathcal{C}$ is an irreducible component of an algebraic subgroup of $\mathbb{G}^2_{m,F}$.
\end{proof}

\begin{rmk}
Corollary \ref{cor:multiplicative} can also be deduced from recent work of Schefer, more specifically from \cite[Theorem 1.2]{Schefer}. This theorem gives an upper bound on the number of $F$-points of order at most $T$ on a closed integral curve in $\mathbb{G}^2_{m,F}$ that is not an irreducible component of an algebraic subgroup. It follows from our hypothesis together with an asymptotic for the number of roots of unity in $F$ of order at most $T$ (see for instance \cite[Theorem 1.4]{Schefer}) that the curve $\mathcal{C}$ in Corollary \ref{cor:multiplicative} has ``too many" points of order at most $T$ and so one concludes again that $\mathcal{C}$ must be an irreducible component of an algebraic subgroup of $\mathbb{G}_{m,F}^2$.
\end{rmk}

From Corollary \ref{cor:multiplicative}, one can readily deduce a solution to both the multiplicative and the cyclotomic support problems for function fields of positive characteristic. One can essentially use the proof of the corresponding support problems for function fields of characteristic $0$, see \cite[Theorem 4.6]{Campagna_Dill}, but now the Frobenius endomorphism has to be taken into account.

\begin{thm}\label{thm:multiplicative}
Let $R$ be the coordinate ring of a smooth affine irreducible curve $\mathcal{C}_{/F}$ and let $A, B \in R\setminus F$. Let $N_0 \in \mathbb{N}$. The
following hold:
\begin{enumerate}
 \item Suppose that for all $n \in \mathbb{N}$ with $n > N_0$, every prime ideal of $R$ that divides $A^n - 1$
also divides $B^n - 1$. Then $B^{p^m} = A^k$ for some $k \in \mathbb{Z}\setminus\{0\}$ and $m \in \mathbb{Z}_{\geq 0}$.
\item Suppose that for all $n \in \mathbb{N}$ with $n > N_0$, every prime ideal of $R$ that divides $\Psi_n(A)$
also divides $\Psi_n(B)$. Then either $B = A^{\pm p^{m}}$ or $A = B^{\pm p^m}$ for some $m \in \mathbb{Z}_{\geq 0}$.
\end{enumerate}
\end{thm}

\begin{proof}
Until the very end of this proof, we treat both cases simultaneously. The tuple $(A,B)$ defines a rational map $\varphi: \mathcal{C} \dashrightarrow \mathbb{G}_{m,F}^2$. Let $\mathcal{C}'$ denote the Zariski closure of the image of $\varphi$. Since $A$ is non-constant, $\mathcal{C}'$ is a curve and $\varphi$ has finite fibers. By \cite[Proposition 15.4~(1)]{GoertzWedhorn}, the set $\mathcal{C}'\setminus \varphi(\mathcal{C})$ is finite. The fact that $A$ is non-constant also implies that the multiplicative order of $A(Q)$ is larger than $N_0$  for all but finitely many $Q \in \mathcal{C}(F)$. Since the maximal ideals of $R$ are in canonical bijection with the $F$-points of $\mathcal{C}$, it follows from either one of our hypotheses that for all but finitely many points $(x_1,x_2) \in \mathcal{C}'(F)$, the multiplicative order of $x_2$ divides the multiplicative order of $x_1$. Corollary \ref{cor:multiplicative} then implies that some relation $A^k B^l = \eta$ holds in $R$ with $k, l \in \mathbb{Z}$ coprime and $\eta$ a root of unity of order $r$, coprime to $p$. Since none of the two coordinate projections is constant on $\mathcal{C}'$, it follows that $kl \neq 0$.

Write $k = k'p^{m_1}$ and $l = l'p^{m_2}$ with $k',l' \in \mathbb{Z}$, $m_1,m_2 \in \mathbb{Z}_{\geq 0}$, and $p \nmid k'l'$. The Zariski closure $\widetilde{\mathcal{C}}$ of the image of $\mathcal{C}'$ under the morphism $[p^{m_1},p^{m_2}]: \mathbb{G}^2_{m,F} \to \mathbb{G}^2_{m,F}$ that sends $(x,y) \in (F^{\ast})^2$ to $(x^{p^{m_1}},y^{p^{m_2}})$ is equal to the closed subscheme of $\mathbb{G}_{m,F}^2$ that is defined by the equation $X^{k'}Y^{l'} = \eta$ where $X,Y$ are the affine coordinates on $\mathbb{G}^2_{m,F}$. Furthermore, in both case (1) as well as case (2), some non-empty open subscheme of the smooth curve $\widetilde{\mathcal{C}}$ also satisfies the respective hypothesis of Theorem \ref{thm:multiplicative} with the same $N_0$ and with $A,B$ equal to the two coordinate projections since the Frobenius endomorphism of $F$ preserves the multiplicative order of a point in $F^{\ast}$. We can then essentially copy the proof of \cite[Theorem 4.6]{Campagna_Dill} to deduce that $\eta = 1$ and $|l'| = 1$ in case (1) while in case (2), we obtain that $\eta = 1$ and $|k'| = |l'| = 1$. The only small adjustment that is necessary is that at any point in the proof of \cite[Theorem 4.6]{Campagna_Dill} where a sufficiently large prime needs to be chosen, one should take a prime different from $p$. This directly yields the corollary in both cases since $k$ and $l$ are coprime and so $m_1m_2 = 0$.
\end{proof}

Note that Theorem \ref{thm:multiplicative}~(2) provides an answer to our finite field variant of Lang's question about roots of unity. Namely, if all but finitely many $F$-points of an integral affine plane curve $\mathcal{C}_{/F}$ have coordinates that are both roots of unity of the same order, then the coordinate functions $X$ and $Y$ on $\mathcal{C}$ satisfy either $X = Y^{\pm p^n}$ or $Y = X^{\pm p^n}$ for some $n \in \mathbb{Z}_{\geq 0}$.

\section{\textsc{Repetitio} -- finite fields and complex multiplication}\label{sec:recordantiae}

From now on, we turn our attention to the finite field variant of Andr\'e's theorem, \emph{i.e.} to Theorem \ref{thm:main_thm}. Before beginning with its proof, we find it convenient to pause for a bit and recall in this short section some well-known facts about elliptic curves that will be frequently used, tacitly or not, in what follows.

Given an algebraically closed field $K$, every element $x \in K$ is the $j$-invariant of some elliptic curve ${(E_x)}_{/K}$. For at most finitely many $x \in K$, the curve $E_x$ is \textit{supersingular}, meaning that its endomorphism ring $\End(E_x)$ has rank $4$ as an abelian group. Suppose now that $K$ is equal to $F$, an algebraic closure of the finite field $\mathbb{F}_p$ for a prime $p$ as above. If $x \in F$ and $E_x$ is not supersingular, then the elliptic curve $E_x$ is called \textit{ordinary}. The endomorphism ring of an ordinary elliptic curve is isomorphic to an order in an  imaginary quadratic field. We will sometimes, by metonymy, call the $j$-invariant $x$ of an ordinary/supersingular elliptic curve $(E_x)_{/F}$ an ordinary/supersingular $j$-invariant. Furthermore, we will say that an element $x$ of an arbitrary algebraically closed field $K$ has CM by the order $\mathcal{O} \not\simeq \mathbb{Z}$ to mean that the curve $E_x$ has complex multiplication by $\mathcal{O}$, \emph{i.e.} that the endomorphism ring of $E_x$ is isomorphic to $\mathcal{O}$.

The next proposition shows that, in the case where $K = F$, applying the Frobenius endomorphism to $x \in K$ does not change its associated CM order.

\begin{prop}\label{prop:Frobenius}
    For every $x \in F$ that is the $j$-invariant of an elliptic curve with complex multiplication by an order $\mathcal{O}$, the element $x^p \in F$ is also the $j$-invariant of an elliptic curve with complex multiplication by $\mathcal{O}$.
\end{prop}

\begin{proof}
    Let $E_{/F}$ be an elliptic curve with $j(E)=x$, say given by some Weierstrass model. Then the curve $E^{(p)}$, whose model is obtained by raising all the coefficients of the model for $E$ to the $p$-th power, has $j$-invariant $j(E^{(p)})=x^p$. The curves $E$ and $E^{(p)}$ are related by the Frobenius homomorphism $(X,Y) \mapsto (X^p,Y^p)$, which is an isogeny of degree $p$ by \cite[Chapter II, Proposition 2.11]{Silverman_book_2009}. The result now follows from \cite[Chapter 13, Theorem 11]{Lang_1987}.
\end{proof}

For the next proposition, recall that the endomorphism algebra of an elliptic curve $E_{/F}$ is defined to be $\End(E) \otimes_{\mathbb{Z}} \mathbb{Q}$.

\begin{prop}\label{prop:isogenous}
    Let $E_{/F}$ and $E'_{/F}$ be two elliptic curves with isomorphic endomorphism algebrae. Then $E$ and $E'$ are isogenous.
\end{prop}

\begin{proof}
    If $E$ (and therefore $E'$) is supersingular, this is \cite[Lemma 42.1.11]{Voight_2021}.
    
    Otherwise, the proposition follows from \cite[Chapter 13, Theorem 14]{Lang_1987} and \cite[Chapter II, Proposition 4.4]{Silverman_book_1994} since any two elliptic curves over $\mathbb{C}$ with CM by orders in the same imaginary quadratic field are isogenous to each other (as one can easily prove using their complex uniformisations).
\end{proof}

\section{\textsc{Probatio} -- proof of Theorem \ref{thm:main_thm}}\label{sec:apologeticum}

Let $F$ be as in Theorem \ref{thm:main_thm} and let $\mathcal{C} \subseteq \mathbb{A}^2_F$ denote a fixed curve satisfying the hypotheses of Theorem \ref{thm:main_thm}, \emph{i.e.} $\mathcal{C}$ is integral and closed and for all but finitely many points $(x_1,x_2) \in \mathcal{C}(F)$ we have that $x_1$ and $x_2$ are $j$-invariants of elliptic curves with isomorphic endomorphism rings. Recall that $X,Y \in F(\mathcal{C})$ denote the elements in the function field of $\mathcal{C}$ induced by the two coordinate projections $\mathcal{C} \to \mathbb{A}^1_F$.

There exists a finite field $\mathbb{F}_q \subset F$ and an integral closed curve $\mathcal{C}_0 \subset \mathbb{A}^2_{\mathbb{F}_q}$ such that $\mathcal{C} = (\mathcal{C}_0)_F$. We obtain a field embedding $\mathbb{F}_q(\mathcal{C}_0) \hookrightarrow F(\mathcal{C})$ and we identify $\mathbb{F}_q(\mathcal{C}_0)$ with its image under this embedding. With this convention, we then have that $X,Y \in \mathbb{F}_q(\mathcal{C}_0)$. We now fix a choice of two elliptic curves $E'_1$ and $E'_2$ defined over $\mathbb{F}_q(\mathcal{C}_0)$ with $j$-invariants $X$ and $Y$ respectively.

For $i=1,2$, we denote by $E'_i[\ell]$ the kernel of multiplication by $\ell$ on $E'_i$ where $\ell \neq p$ is an arbitrary fixed odd prime. We choose a finite extension $K$ of $\mathbb{F}_q(\mathcal{C}_0)$ over which every point of $E'_i[\ell]$ is rational and we fix a separable closure $K^{\mathrm{sep}}$ of $K$. We set $E_i = (E'_i)_K$ for $i = 1,2$.

\begin{prop}\label{prop:reduction}
In the setting and under the hypotheses described in this section so far, the elliptic curves $E_1$ and $E_2$ are geometrically isogenous.
\end{prop}

\begin{proof}
Let $T_\ell(E_i)$ denote the $\ell$-adic Tate module of $E_i$ over $K$ ($i = 1,2$). Fixing isomorphisms $T_\ell(E_i) \simeq \mathbb{Z}_\ell^2$, we obtain continuous representations $\rho_i: \mathrm{Gal}(K^{\mathrm{sep}}/K) \to \mathrm{GL}(T_\ell(E_i) \otimes_{\mathbb{Z}_\ell} \mathbb{Q}_\ell) \simeq \mathrm{GL}_2(\mathbb{Q}_\ell)$ for $i = 1,2$.

It follows from the fact that $\mathcal{C}$ satisfies the hypothesis of Theorem \ref{thm:main_thm} that for all primes $\mathfrak{p}$ of $K$ outside a finite set $S$, the elliptic curves $E_1$ and $E_2$ have good reduction modulo $\mathfrak{p}$ and the reductions of $E_1$ and $E_2$ modulo $\mathfrak{p}$ have isomorphic geometric endomorphism rings and are therefore geometrically isogenous by Proposition \ref{prop:isogenous}. Since $E_1$ and $E_2$ have full $\ell$-torsion over $K$, their reductions modulo a prime $\mathfrak{p} \not\in S$ have full $\ell$-torsion over the residue field $k_{\mathfrak{p}}$ of $\mathfrak{p}$ because of \cite[Proposition 20.7]{MilneAV}. By \cite[Theorem 2.4]{Silverberg_1992}, the reductions are therefore actually isogenous, not only geometrically isogenous.

For each prime $\mathfrak{p} \not\in S$ and each prime $\mathfrak{P}$ of $K^{\mathrm{sep}}$ lying above $\mathfrak{p}$, let $D(\mathfrak{P}|\mathfrak{p}) \subset \mathrm{Gal}(K^{\mathrm{sep}}/K)$ denote the corresponding decomposition group and fix an element $\mathrm{Frob}_{\mathfrak{P}} \in D(\mathfrak{P}|\mathfrak{p})$ that reduces to the Frobenius relative to $k_{\mathfrak{p}}$. Since \cite[Proposition 20.7]{MilneAV} implies that reduction modulo $\mathfrak{P}$ induces a group isomorphism from the $\ell^n$-torsion of $(E_i)_{K^{\mathrm{sep}}}$ to the $\ell^n$-torsion of the reduction of $(E_i)_{K^{\mathrm{sep}}}$ modulo $\mathfrak{P}$ for every $n \in \mathbb{N}$ and $i = 1,2$, it then follows from \cite[Theorem 1(c)]{Tate_1966} that
\[ \mathrm{Tr}(\rho_1(\mathrm{Frob}_{\mathfrak{P}})) = \mathrm{Tr}(\rho_2(\mathrm{Frob}_{\mathfrak{P}})),\]
where $\mathrm{Tr}$ denotes the matrix trace.

By the Chebotarev density theorem for function fields \cite[Theorem 9.13A]{Rosen_2002}, the set of all $\mathrm{Frob}_{\mathfrak{P}}$ as in the last paragraph for varying $\mathfrak{p} \not\in S$ and $\mathfrak{P}$ lying above $\mathfrak{p}$ is dense in $\mathrm{Gal}(K^{\mathrm{sep}}/K)$. Since $\rho_1$ and $\rho_2$ are continuous, it follows that $\mathrm{Tr}(\rho_1(\sigma)) = \mathrm{Tr}(\rho_2(\sigma))$ for all $\sigma \in \mathrm{Gal}(K^{\mathrm{sep}}/K)$.

We now deduce from \cite[Chapitre XII, Th\'eor\`eme 2.5~(ii)]{Moret_Bailly_Asterisque} and \cite[Chapter XVII, Corollary 3.8]{Lang_Algebra} that the representations $\rho_1$ and $\rho_2$ are isomorphic. It then follows from \cite[Chapitre XII, Th\'eor\`eme 2.5~(i)]{Moret_Bailly_Asterisque} that $E_1$ and $E_2$ are geometrically isogenous.
\end{proof}

To finish the proof of Theorem \ref{thm:main_thm}, we will need the following lemma on specialisation of isogenies between elliptic curves. This is certainly well-known to the experts, but we include a proof for lack of a suitable reference.

\begin{lem}\label{lem:isogenyspecialisation}
Let $\mathcal{D}$ be an integral curve over an algebraically closed field $L$ of characteristic $p \geq 0$. Fix an algebraic closure $\overline{L(\mathcal{D})}$ of $L(\mathcal{D})$. Let $(E_1)_{/\overline{L(\mathcal{D})}}$ and $(E_2)_{/\overline{L(\mathcal{D})}}$ be two elliptic curves whose $j$-invariants $j_1$ and $j_2$ belong to $L(\mathcal{D})$ and suppose that there exists a cyclic isogeny $\chi: E_1 \to E_2$ of degree $d$ such that $p \nmid d$. Then for all but finitely many points $x \in \mathcal{D}(L)$, the curve $\mathcal{D}$
 is smooth at $x$, the point $x$ is a pole of neither $j_1$ nor $j_2$, and there exists a cyclic isogeny of degree $d$ from an elliptic curve $\left(E_{j_1(x)}\right)_{/L}$ with $j$-invariant $j_1(x)$ to an elliptic curve $\left(E_{j_2(x)}\right)_{/L}$ with $j$-invariant $j_2(x)$.
\end{lem}

\begin{proof}
The elliptic curves $E_1$ and $E_2$, the isogeny $\chi$, and the $d$-torsion of $E_1$ are all defined over a finite field extension of $L(\mathcal{D})$. This finite field extension is the function field $L(\widetilde{\mathcal{D}})$ of a smooth irreducible curve $\widetilde{\mathcal{D}}_{/L}$, which we choose in such a way that the inclusion $L(\mathcal{D}) \subseteq L(\widetilde{\mathcal{D}})$ induces a non-constant morphism $\widetilde{\mathcal{D}} \to \mathcal{D}$. In the following, we will identify $L(\widetilde{\mathcal{D}})$ with a subfield of $\overline{L(\mathcal{D})}$.

It follows from \cite[Theorem 1.4/3, Lemma 7.3/1, and Proposition 7.3/6]{BLR} that, after maybe replacing $\widetilde{\mathcal{D}}$ by a non-empty open subscheme, there exist two elliptic schemes $\mathcal{E}_1$ and $\mathcal{E}_2$ over $\widetilde{\mathcal{D}}$ and a homomorphism $\Psi: \mathcal{E}_1 \to \mathcal{E}_2$ of group schemes over $\widetilde{\mathcal{D}}$ such that $(\mathcal{E}_i)_{\overline{L(\mathcal{D})}} = E_i$ for $i = 1,2$, $\Psi_{\overline{L(\mathcal{D})}} = \chi$, $(\ker \chi)(\overline{L(\mathcal{D})}) = (\ker \Psi)(L(\widetilde{\mathcal{D}}))$, and the restriction of $\Psi$ to each fiber is finite, flat, and surjective. Note that the two morphisms from $\widetilde{\mathcal{D}}$ to the $j$-line 
$\mathbb{A}^1_L$ induced by the $j$-invariants of $E_1$ and $E_2$ factor as the compositions of $\widetilde{\mathcal{D}} \to \mathcal{D}$ with the respective rational map $j_i: \mathcal{D} \dashrightarrow \mathbb{A}^1_L$.

By \cite[Corollary 14.27]{GoertzWedhorn}, $\Psi$ is flat and so $\ker \Psi$ is flat over $\widetilde{\mathcal{D}}$. In particular, all its irreducible components dominate $\widetilde{\mathcal{D}}$ and are $1$-dimensional by \cite[Propositions 14.14 and 14.109~(2)]{GoertzWedhorn}. Let $[d]: \mathcal{E}_1 \to \mathcal{E}_1$ denote multiplication by $d$. By the N\'eron mapping property, $\ker \Psi$ is a closed subscheme of $\ker [d]$ since the restriction of $[d]$ to the generic fiber factors through the corresponding restriction of $\Psi$. Furthermore, $\ker [d] \to \widetilde{\mathcal{D}}$ is finite and \'etale by \cite[Proposition 20.7]{MilneAV} since $p \nmid d$. By our choice of $\widetilde{\mathcal{D}}$, the absolute Galois group of $L(\widetilde{\mathcal{D}})$ acts trivially on the generic fiber of $\ker [d]$ and it follows from \cite[Chapter I, Theorem 5.3]{Milne_book_1980} that $\ker [d]$ is a constant group scheme over $\widetilde{D}$. We deduce that $\ker \Psi$ is a constant group scheme over $\widetilde{D}$ as well.

It follows that for all $y \in \widetilde{\mathcal{D}}(L)$, there is a cyclic isogeny $\Psi_y: (\mathcal{E}_1)_y \to (\mathcal{E}_2)_y$ of degree $d$. Since the image of the non-constant morphism $\widetilde{\mathcal{D}} \to \mathcal{D}$ is cofinite in $\mathcal{D}$ by \cite[Proposition 15.4~(1)]{GoertzWedhorn}, since 
$\mathcal{D}$ contains a smooth open dense subscheme by \cite[Theorem 6.19]{GoertzWedhorn}, and since $j_1$ and $j_2$ both have at most finitely many poles on this open dense subscheme, this implies that for all but finitely many $x \in \mathcal{D}(L)$, the curve $\mathcal{D}$ is smooth at $x$, the point $x$ is a pole of neither $j_1$ nor $j_2$, and there is a cyclic isogeny of degree $d$ from an elliptic curve over $L$ with $j$-invariant $j_1(x)$ to an elliptic curve over $L$ with $j$-invariant $j_2(x)$.
\end{proof}

Because of Proposition \ref{prop:reduction}, the following theorem will imply Theorem \ref{thm:main_thm}.

\begin{thm}
    Recall that $F$ and $\mathcal{C}$ are as in the hypotheses of Theorem \ref{thm:main_thm} and that $X,Y \in F(\mathcal{C})$ denote the elements in the function field of $\mathcal{C}$ induced by the two coordinate projections $\mathcal{C} \to \mathbb{A}^1_F$. Suppose that some elliptic curve over $F(\mathcal{C})$ with $j$-invariant $X$ is geometrically isogenous to an elliptic curve over $F(\mathcal{C})$ with $j$-invariant $Y$. Then there exists $n \in \mathbb{Z}_{\geq 0}$ such that either $X = Y^{p^n}$ or $Y = X^{p^n}$.
\end{thm}

\begin{proof}
Fix an algebraic closure $\overline{F(\mathcal{C})}$ of $F(\mathcal{C})$ and let $(E_1)_{/\overline{F(\mathcal{C})}}$ and $(E_2)_{/\overline{F(\mathcal{C})}}$ denote elliptic curves with $j$-invariants $X$ and $Y$ respectively. By hypothesis, there exists an isogeny $\varphi: E_1 \to E_2$. Choosing $\varphi$ such that $\deg \varphi$ is minimal among the degrees of all isogenies from $E_1$ to $E_2$, we can assume without loss of generality that there exists no decomposition $\varphi = \psi \circ [N]$ where $[N]$ denotes the multiplication by $N > 1$ on $E_1$ and $\psi$ is another isogeny from $E_1$ to $E_2$. In particular, $(\ker \varphi)(\overline{F(\mathcal{C})})$ is cyclic.

Our hypothesis on $\mathcal{C}$ implies that neither $E_1$ nor $E_2$ is isotrivial, \emph{i.e.} that neither $j(E_1)$ nor $j(E_2)$ belongs to $F$. By \cite[Proposition 4.8]{Pazuki_Griffon}, one among $\varphi$ and its dual $\widehat{\varphi}$ is separable since $\varphi$ does not factor via any multiplication-by-$N$ map for $N > 1$. In particular, by \cite[Proposition 4.7]{Pazuki_Griffon}, we can assume, up to switching $X$ and $Y$ and using $\widehat{\varphi}$ instead of $\varphi$, that there exists a cyclic isogeny $\chi: E_{1,n} \to E_2$ which is separable with separable dual where $(E_{1,n})_{/\overline{F(\mathcal{C})}}$ denotes an elliptic curve with $j$-invariant $X^{p^n}$. By \cite[Lemma 4.5]{Pazuki_Griffon}, the isogeny $\chi$ has degree coprime to $p$.

Set $d = \deg \chi$. If $d = 1$, then $\chi$ is an isomorphism, so $X^{p^n} = Y$ and we are done. Hence, we assume from now on that $d > 1$ and we aim to obtain a contradiction.

Since $d > 1$, there exists some prime $\ell$ dividing $d$. We have $\ell \neq p$ since $p \nmid d$. By Dirichlet's theorem on primes in arithmetic progressions, there exist infinitely many discriminants $-D < -d$ such that
\begin{itemize}
    \item $D$ is prime (in particular, $-D \equiv 1 \text{ mod } 4$),
    \item $\ell$ is inert in $\mathbb{Q}(\sqrt{-D})$, and
    \item $p$ is split in $\mathbb{Q}(\sqrt{-D})$.
\end{itemize}
The third condition together with \cite[Chapter 13, Theorem 12]{Lang_1987} ensures that, for each such discriminant $-D$, there exists some $x \in F$ with CM by the quadratic order of discriminant $-D$. 

From this together with \cite[Proposition 15.4~(1)]{GoertzWedhorn}, applied to the first coordinate projection $\mathcal{C} \to \mathbb{A}^1_F$ (which is non-constant), we now deduce that, for any discriminant $-D$ as above with $D > 0$ sufficiently large, there exists a point $(x_1,x_2) \in \mathcal{C}(F)$ such that $x_1$ has CM by the quadratic order of discriminant $-D$. After making $D$ even larger if necessary, we can also assume, thanks to Lemma \ref{lem:isogenyspecialisation} and the property of the curve $\mathcal{C}$ in the statement of the theorem, that
\begin{itemize}
    \item $x_2$ has CM by the quadratic order of discriminant $-D$ as well and
    \item there is an isogeny of degree $d$ with cyclic kernel from an elliptic curve over $F$ with $j$-invariant $x_1^{p^n}$ to an elliptic curve over $F$ with $j$-invariant $x_2$.
\end{itemize}

Proposition \ref{prop:Frobenius} then implies that also the elliptic curve over $F$ with $j$-invariant $x_1^{p^n}$ has CM by the maximal order in $\mathbb{Q}(\sqrt{-D})$. Hence, there exists an isogeny of degree $d$ with cyclic kernel between two elliptic curves over $F$ with CM by the quadratic order of discriminant $-D$, where $d$ and $-Dp$ are coprime and $d$ is divisible by the prime $\ell$, which is inert in $\mathbb{Q}(\sqrt{-D})$. This however contradicts \cite[Proposition 5.1]{Campagna_Dill} and we are done.
\end{proof}

\section{\textsc{Peroratio}} \label{sec:final section}

We begin this short final section by proving Theorem \ref{thm:main_theorem}.

\begin{proof}[Proof of Theorem \ref{thm:main_theorem}]
Let $\mathcal{D} \subseteq \mathbb{A}^2_F$ be the schematic image of the morphism $\varphi: \mathcal{C} \to \mathbb{A}^2_F$ defined by the pair $(A,B)$. Then $\mathcal{D}$ is a curve that is irreducible (because it is the schematic image of an irreducible curve) and reduced by \cite[Remark 10.32]{GoertzWedhorn}. Thus, $\mathcal{D}$ is integral and in order to conclude we only have to prove that $\mathcal{D}$ satisfies the main hypothesis of Theorem \ref{thm:main_thm}, \textit{i.e.} that all but finitely many points in $\mathcal{D}(F)$ have coordinates that are $j$-invariants of elliptic curves with isomorphic endomorphism rings.

By \cite[Proposition 15.4~(1)]{GoertzWedhorn} applied to the non-constant morphism $\varphi$, the set $\mathcal{D}(F) \setminus \varphi(\mathcal{C}(F))$ is finite. By hypothesis, there exists some $D_0 \in \mathbb{D}$ such that the implication ``$\mathfrak{p} \mid H_D(A) \Rightarrow \mathfrak{p} \mid H_D(B)$" holds for all $D < D_0$ and all prime ideals $\mathfrak{p}$ of $R$. Let $\mathcal{S} \subseteq \mathcal{C}(F)$ be the subset of points $Q \in \mathcal{C}(F)$ such that $A(Q)$ is ordinary and $H_D(A(Q)) \neq 0$ for all the finitely many $D \in \mathbb{D}$ with $D \geq D_0$. Since $A \in R \setminus F$ and all but finitely many elements of $F$ are ordinary, the set $\mathcal{C}(F) \setminus \mathcal{S}$ is finite and so also $\mathcal{D}(F) \setminus \varphi(\mathcal{S})$ is finite.

Let now $(x,y)$ be a point in $\varphi(\mathcal{S})$. By \cite[Chapter 13, Theorem 14]{Lang_1987}, there exists $D \in \mathbb{D}$ such that $H_D(x) = 0$. Since $(x,y) \in \varphi(\mathcal{S})$, there is a point $Q \in \mathcal{S}$ such that $\varphi(Q) = (x,y)$ and hence $A(Q) = x$. In particular, $x$ is ordinary. Moreover, it follows that $H_D(A(Q)) = 0$, which, by definition of $\mathcal{S}$, implies that $D < D_0$. Therefore, we can apply the hypothesis to deduce that also $H_D(y) = H_D(B(Q)) = 0$. Because of the irreducibility of $H_D(T) \in \mathbb{Z}[T]$, all zeroes of $H_D$ in $\overline{\mathbb{Q}}$ have CM by the imaginary quadratic order of discriminant $D$. Combined with \cite[Chapter 13, Theorem 12]{Lang_1987} and the fact that $H_D(x) = H_D(y) = 0$, this implies that $y$ has CM by the same order as $x$.

Since $\mathcal{D}(F) \setminus \varphi(\mathcal{S})$ is finite, this shows that the curve $\mathcal{D}$ satisfies the main hypothesis of Theorem \ref{thm:main_thm} and we are done.

\end{proof}

We conclude the article by showing that ``all but finitely many" cannot be replaced by ``infinitely many" in the hypothesis of Theorem \ref{thm:main_thm}.

\begin{prop}\label{prop:infinitelymany}
Let $F$ be as in Theorem \ref{thm:main_thm}. For any integral closed curve $\mathcal{C} \subseteq \mathbb{A}^2_F$ that is neither a vertical nor a horizontal line, there are infinitely many points $(x,y) \in \mathcal{C}(F)$ for which $x$ and $y$ both have the same multiplicative order and such that $x$ and $y$ are $j$-invariants of elliptic curves with isomorphic endomorphism rings.
\end{prop}

Note that, since there are only finitely many elements of $F$ with a prescribed multiplicative order and, similarly, only finitely many elements of $F$ having CM by a fixed endomorphism ring, it is necessary to exclude vertical and horizontal lines in the hypothesis of Proposition \ref{prop:infinitelymany}.

\begin{proof}
Let $\mathcal{C} \subseteq \mathbb{A}^2_F$ be an integral closed curve as in the statement of the proposition. The idea is to prove that there are infinitely many points $(x,y) \in \mathcal{C}(F)$ such that $x=y^{p^n}$ for some $n \in \mathbb{N}$. Since the Frobenius endomorphism of $F$ preserves multiplicative orders as well as the property of having CM by a fixed order $\mathcal{O}$ (see Proposition \ref{prop:Frobenius}), this suffices to conclude. An analogous argument already appears in the proof of \cite[Proposition 3.6]{Campagna_Dill} and we sketch it again here for completeness.

Let $K$ be the function field of $\mathcal{C}$ and let $X,Y \in K$ be the coordinate functions on $\mathcal{C}$. Since $X \not\in F$, there exists $X_0 \in K$ such that $X_0$ is not a $p$-th power in $K$ and $X = X_0^{p^m}$ for some $m \in \mathbb{Z}_{\geq 0}$. For $n \in \mathbb{N}$, let $P_n=X_0-Y^{p^n}$ so that $X_0-Y^{p^n}-P_n=0$. Let $\overline{\mathcal{C}}_{/F}$ be a smooth projective irreducible curve whose function field is isomorphic to $K$. There exists an open immersion $\mathcal{C}^{\circ} \hookrightarrow \overline{\mathcal{C}}$ where $\mathcal{C}^{\circ}$ denotes the smooth locus of $\mathcal{C}$. We identify $\mathcal{C}^{\circ}$ with its image in $\overline{\mathcal{C}}$. Let $S \subseteq \overline{\mathcal{C}}(F)$ be the union of the finite set $(\overline{\mathcal{C}}\setminus\mathcal{C}^{\circ})(F)$ and the set of points in $\mathcal{C}^{\circ}(F)$ where the order of $X_0$ or $Y$, regarded as rational functions on $\mathcal{C}^{\circ}$, is non-zero. Since $X_0$ is not a $p$-th power in $K$, the function $X_0/(-Y^{p^n})$ is also not a $p$-th power in $K$. Since $Y$ is non-constant, the function $X_0/(-Y^{p^n})$ is non-constant for $n$ large enough and its degree goes to infinity as $n$ goes to infinity (in particular, $P_n \neq 0$ for all sufficiently large $n$). Hence, the $ABC$ theorem for function fields \cite[Lemma 10 on page 97]{Mason} shows that, if $n$ is large enough, there exists a point $Q=(x,y) \in \overline{\mathcal{C}}(F) \setminus S \subseteq \mathcal{C}^{\circ}(F)$ such that $P_n(Q)=0$. This implies that $X_0(Q) =y^{p^n}$ and hence $x = X_0(Q)^{p^m} = y^{p^{m+n}}$. By adding the point $Q$ to the exceptional set $S$ and repeating the argument, one constructs infinitely many points with the desired property. This concludes the proof and the article.
    
\end{proof}

\section*{\textsc{Gratiarvm actio}}
We thank Philipp Habegger, Jonathan Pila, and Thomas Scanlon for helpful discussions and we thank the Mathematical Sciences Research Institute for the stimulating and inspiring environment where many ideas contained in this paper have seen the light of day. We thank GD's parents for useful consultations on Latin vocabulary.

FC is supported by ANR-20-CE40-0003 Jinvariant.

When this project began, GD was supported by the Swiss National Science Foundation through the Early Postdoc.Mobility grant no. P2BSP2\_195703. He thanks the Mathematical Institute of the University of Oxford and his host there, Jonathan Pila, for hosting him as a visitor for the duration of this grant. This material is based upon work supported by the National Science Foundation under Grant No.~DMS--1928930 while GD was in residence at the Mathematical Sciences Research Institute in Berkeley, California, during the Spring 2023 semester. GD thanks the DFG for its support (grant no. EXC-2047/1 - 390685813).

\vspace{\baselineskip}
\noindent
\framebox[\textwidth]{
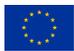
\begin{tabular*}{0.96\textwidth}{@{\extracolsep{\fill} }cp{0.84\textwidth}}
\raisebox{-0.7\height}{%
    \begin{tikzpicture}[y=0.80pt, x=0.8pt, yscale=-1, inner sep=0pt, outer sep=0pt, 
    scale=0.12]
    \definecolor{c003399}{RGB}{0,51,153}
    \definecolor{cffcc00}{RGB}{255,204,0}
    \begin{scope}[shift={(0,-872.36218)}]
      \path[shift={(0,872.36218)},fill=c003399,nonzero rule] (0.0000,0.0000) rectangle (270.0000,180.0000);
      \foreach \myshift in 
           {(0,812.36218), (0,932.36218), 
    		(60.0,872.36218), (-60.0,872.36218), 
    		(30.0,820.36218), (-30.0,820.36218),
    		(30.0,924.36218), (-30.0,924.36218),
    		(-52.0,842.36218), (52.0,842.36218), 
    		(52.0,902.36218), (-52.0,902.36218)}
        \path[shift=\myshift,fill=cffcc00,nonzero rule] (135.0000,80.0000) -- (137.2453,86.9096) -- (144.5106,86.9098) -- (138.6330,91.1804) -- (140.8778,98.0902) -- (135.0000,93.8200) -- (129.1222,98.0902) -- (131.3670,91.1804) -- (125.4894,86.9098) -- (132.7547,86.9096) -- cycle;
    \end{scope}
    \end{tikzpicture}%
}
&
Francesco Campagna and Gabriel Dill have received funding from the European Research Council (ERC) under the European Union’s Horizon 2020 research and innovation programme (grant agreement n$^\circ$ 945714).
\end{tabular*}
}

\printbibliography

\end{document}